\documentclass[10pt]{amsart}
\usepackage{amsmath,amsthm,amsfonts,amssymb,amscd}
\usepackage[small,nohug,heads=LaTeX]{diagrams}
\diagramstyle[labelstyle=\scriptstyle]
\theoremstyle{plain}
\newtheorem{thm}[subsection]{Theorem}
\newtheorem{defn}[subsection]{Definition}

\theoremstyle{definition}

\newtheorem{rmk}[subsection]{Remark}
\newtheorem{eg}[subsection]{Example}
\numberwithin{equation}{section}

\title[Stinespring's Theorem for maps on Hilbert $C^*$-modules]{Stinespring's Theorem for maps on Hilbert $C^*$- modules}
\author{B. V. Rajarama Bhat}
\author{G. Ramesh}
\author{K. Sumesh}
\address{Stat Math Unit\\I. S. I. Bangalore, Bangalore\\ India-560 059.}
\email{bhat@isibang.ac.in,ramesh@isibang.ac.in,sumesh@isibang.ac.in}
\thanks{The first author is thankful to UKIERI for financial support and the second author is thankful to the NBHM for financial support and ISI Bangalore for providing necessary facilities to carry out this work.}
\date{\today.}
\begin{document}
\begin{abstract}
We strengthen  Mohammad B. Asadi's  analogue of Stinespring's
theorem for certain maps on Hilbert $C^*$-modules. We also show
that any two minimal Stinespring representations are unitarily
equivalent. We illustrate the main theorem with an example.
\end{abstract}
\maketitle
\section{Introduction}
Stinespring's representation theorem is a fundamental theorem in
the theory of completely positive maps. It is a structure theorem
for completely positive maps from a $C^*$-algebra into the algebra
of bounded operators on a Hilbert space. This theorem provides a
representation for  completely positive maps, showing that they
are simple modifications of  $*$-homomorphisms ( see
\cite{wfstinespring} for details). One may consider it as  a
natural generalization of the well-known Gelfand-Naimark-Segal
theorem for states on $C^*$-algebras (see \cite[Theorem 4.5.2,
page 278]{ringkad1} for details). Recently, a theorem which looks
like Stinespring's  theorem was presented by Mohammad B. Asadi in
\cite{asadi} for a class of unital maps on Hilbert $C^*$-modules.
Here we strengthen this result by removing a   technical condition
of Asadi's theorem \cite{asadi}. We also remove the assumption of
unitality on maps under consideration.  Further we prove
uniqueness up to unitary equivalence for  minimal representations,
which is an important ingredient of structure theorems  like GNS
theorem and  Stinespring's theorem. Now the result looks even more
like Stinespring's theorem.
\subsection{Notation and Earlier Results}
We denote Hilbert spaces by $H,H_1,H_2$ etc and the corresponding
inner product and the induced norm by $\langle ., .\rangle$ and
$\|\cdot \|$ respectively. Throughout we assume that the inner
product is conjugate linear in the first variable and linear in
the second variable. The space of bounded linear operators from
$H_1$ to $H_2$ is denoted by $\mathcal B(H_1,H_2)$ and $\mathcal
B(H,H)=\mathcal B(H)$. We denote $C^*$-algebras by $\mathcal
A,\mathcal B$ etc. The $C^*$-algebra  of all $n\times n$ matrices
with entries from $\mathcal A$ is denoted by $\mathcal
M_n(\mathcal A)$. If $L$ is a subset of a Hilbert space, then
$[L]:=\overline{\text{span}}(L).$

A linear map  $\phi:\mathcal A\rightarrow  \mathcal B$ is said to
be \textbf{positive} if $\phi(a^*a)\geq 0,\; \text{for all}\; a\in
\mathcal A$. If $\phi_n:\mathcal M_n(\mathcal A)\rightarrow
\mathcal M_n(\mathcal B)$, given by
$\phi_n((a_{ij}))=(\phi(a_{ij})),\; i,j=1,2,\dots, n$ is positive,
then $\phi$ is said to be $n$-positive. If $\phi_n$ is positive
for all $n$ ($n\geq 1$), then $\phi$ is called a
\textbf{completely positive map}. Completely positive maps from a
$C^*$-algebra $\mathcal A$ to $\mathcal B(H)$ is characterized by
Stinespring in \cite{wfstinespring}. This fundamental theorem is
well-known as Stinespring's representation theorem.
\begin{thm}(Stinespring's representation theorem \cite[Theorem 4.1, page
43]{paulsencbmoa}). Let $\mathcal A$ be a unital $C^*$-algebra and
$\phi:\mathcal A\rightarrow \mathcal B(H)$ be a completely
positive map. Then there exists a Hilbert space $K$, a unital
$\ast$-homomorphism $\rho:\mathcal A\rightarrow \mathcal B(K)$ and
a bounded operator $V:H\rightarrow K$ with $\|\phi(1)\|=\|V\|^2$
such that
$$\phi(a)=V^*\rho(a)V, \quad \text{for all}\; a\in \mathcal A.$$
\end{thm}
 The triple $(\rho, V, K)$ in the Stinespring's representation theorem is called a representation for $\phi$. If $[\rho(\mathcal A)VH]=K$, then it is called a minimal representation.
It is known that if $(\rho,V,K)$ and $(\rho',V',K')$ are two
minimal Stinespring representations for $\phi$, then there exists
a unitary operator $U:K\rightarrow K'$ such that $UV=V'$ and
$U\rho(a) U^*=\rho'(a) $ for all $a\in \mathcal A$ (see
\cite[Proposition 4.2, page 46]{paulsencbmoa}).



Now we consider maps on Hilbert $C^*$-modules. Let $E$ be a
Hilbert $C^*$-module over a $C^*$ algebra $\mathcal A$ (see
\cite{lance} for details of Hilbert $C^*$-modules). Let
$\phi:\mathcal A\rightarrow \mathcal B(H_1)$ be linear. Then
$\phi$ is said to be a morphism if it is a $\ast$ homomorphism and
nondegenerate  (i.e., $\overline{\phi(\mathcal A)H_1}=H_1$). We
remind the reader that  $\mathcal B(H_1,H_2)$ is a Hilbert
$\mathcal B(H_1)$-module for any two Hilbert spaces $H_1, H_2$,
with the following operations:
\begin{enumerate}
 \item  module map: $(T,S)\mapsto TS: \mathcal B(H_1,H_2)\times \mathcal B(H_1)\rightarrow \mathcal B(H_1,H_2)$
\item  inner product: $\langle T,S\rangle \mapsto T^*S :\mathcal
B(H_1,H_2)\times \mathcal B(H_1,H_2)\rightarrow \mathcal B(H_1)$
\end{enumerate}
A map $\Phi:E\rightarrow \mathcal B(H_1, H_2)$ is said to be a
\begin{enumerate}
 \item \textbf{$\phi$-map} if $\langle \Phi(x),\Phi(y)\rangle =\phi(\langle x,y\rangle)$ for all $x,y\in E$;
\item \textbf{$\phi$-morphism} if $\Phi$ is a $\phi$-map and
$\phi$ is a morphism; \item \textbf{$\phi$-representation} if
$\Phi$ is a $\phi$-morphism and $\phi$ is a representation.
\end{enumerate}
Note that a $\phi$-morphism $\Phi$ is linear and satisfies
$\Phi(xa)=\Phi(x)\phi(a)$ for every $x\in E$ and $a\in \mathcal
A$.
\begin{thm}(Mohammad B. Asadi \cite{asadi})\label{asdithm}.
 If $E$ is a Hilbert $C^*$-module over the unital $C^*$-algebra $\mathcal A$, and $\phi:\mathcal A\rightarrow \mathcal B(H_1)$ is a completely positive map with $\phi(1)=1$
and $\Phi:E\rightarrow \mathcal B(H_1,H_2)$ is a $\phi$-map with
the additional property $\Phi(x_0)\Phi(x_0)^*=I_{H_2}$ for some
$x_0\in E$, where $H_1,H_2$ are Hilbert spaces, then there exist
Hilbert spaces $K_1,K_2$ and isometries $V:H_1\rightarrow K_1$ and
$W:H_2\rightarrow K_2$ and a $\ast$-homomorphism $\rho:\mathcal
A\rightarrow \mathcal B(K_1)$ and a $\rho$-representation
$\Psi:E\rightarrow \mathcal B(K_1,K_2)$ such that
\begin{equation*}
 \phi(a)=V^*\rho(a)V, \quad \Phi(x)=W^*\Psi(x)V
\end{equation*}
for all $x\in E,\; a\in \mathcal A$.
\end{thm}

The proof of this Theorem as given in \cite{asadi} is erroneous as
the sesquilinear form defined there on $E\otimes H_2$ is not
positive definite. This can be fixed by inter-changing the indices
$i, j$ in the definition of this form. However such a modification
yields a `non-minimal' representation. Moreover, the technical
condition to have  $\Phi(x_0)\Phi(x_0)^*=I_{H_2}$ for some $x_0\in
E$ is completely unnecessary.
\section{Main Results}
In this Section we strengthen Asadi's theorem for a $\phi$-map
$\Phi$  and discuss the minimality of the representations.
\begin{thm}\label{stinespringhilbertmodules}
Let  $\mathcal A$ be a unital $C^*$-algebra and  $\phi:\mathcal A\rightarrow \mathcal B(H_1)$ be a completely positive map. Let $E$ be a
Hilbert $\mathcal A$-module and $\Phi:E\rightarrow \mathcal B(H_1,H_2)$ be a $\phi$-map. Then there
exists a pair of triples $((\rho,V,K_1), (\Psi,W,K_2))$, where
\begin{enumerate}
        \item   $K_1$  and $K_2$ are Hilbert spaces;
\item  $\rho:\mathcal A\rightarrow \mathcal B(K_1)$ is a unital
$\ast$-homomorphism and $\Psi:E\rightarrow \mathcal B(K_1,K_2)$ is
a $\rho$-morphism; \item  $V:H_1\rightarrow K_1$ and
$W:H_2\rightarrow K_2$ are bounded linear operators;
 \end{enumerate}
 such that
$$ \phi(a)=V^*\rho(a)V,\;\; \text{for all}\;\; a\in \mathcal A \;\; \text{and}\;\; \Phi(x)=W^*\Psi(x)V,\;\; \text{for all}\;\; x\in E.$$
 \end{thm}
\begin{proof}
 We prove the theorem in two steps.\\

\noindent{Step I}: Existence of $\rho,V$ and $K_1$: This is the
content of  Stinespring's theorem \cite[Theorem 4.1, page
43]{paulsencbmoa} as  $\phi$ is a completely positive map. In fact
we can choose a minimal Stinespring representation for  $\phi$.
In this case $K_1=[\rho(\mathcal A) VH_1]$.\\
%

\noindent{Step II}: Construction of  $\Psi,W$ and $K_2$: Let $K_2:=[\Phi(E)H_1]$. Now define $\Psi: E\rightarrow \mathcal B(K_1,K_2)$ as follows:\\
 For $x\in E$, define $\Psi(x):K_1\rightarrow K_2$ by
\begin{equation*}
\Psi(x)\big(\sum_{j=1}^n \rho(a_j)Vh_j\big):=\sum_{j=1}^n\Phi(xa_j)h_j,\quad  a_j\in \mathcal A,\; h_j\in H_1,\; j=1,\dots, n,\, n\geq 1.
\end{equation*}
First we claim that $\Psi(x)$ is well defined. Let $a_j\in \mathcal A, h_j\in H_1,\; j=1,2,\dots, n, \; n\geq 1$. Then we have
$$\begin{array}{ll}
||\Psi(x)\big(\sum_{j=1}^n \rho(a_j)Vh_j\big)||^2 &=\left\langle \displaystyle \sum_{j=1}^n \Phi(xa_j)h_j,\sum_{i=1}^n \Phi(xa_i)h_i\right\rangle \\
                                                      &=\displaystyle \sum_{i,j=1}^n \langle \Phi(xa_j)h_j, \Phi(xa_i)h_i\rangle\\
                                                      &=\displaystyle \sum_{i,j=1}^n \langle h_j, (\Phi(xa_j))^*\Phi(xa_i)h_i\rangle\\
                                                      &=\displaystyle \sum_{i,j=1}^n \langle h_j, \phi(a_j^*\langle x,x\rangle a_i)h_i\rangle\\
                                                     &=\displaystyle \sum_{i,j=1}^n \langle h_j, V^*\rho(a_j^*\langle x,x\rangle a_i)Vh_i\rangle\\
                                                     &=\displaystyle \sum_{i,j=1}^n \langle \rho(a_j)Vh_j, \rho( \langle x,x\rangle ) \rho(a_i)Vh_i\rangle\\
                                                     &=\displaystyle \left\langle \sum_{j=1}^n \rho(a_j)Vh_j, \rho( \langle x,x\rangle ) (\sum_{i=1}^n \rho(a_i)Vh_i)\right\rangle\\
                                                     &\leq \|\rho(\langle x,x\rangle)\|\, \|\sum_{j=1}^n\rho(a_j)Vh_j\|^2\\
                                                      &\leq \|x\|^2 \, \|\displaystyle \sum_{j=1}^n\rho(a_j)Vh_j\|^2.
\end{array}$$
Hence $\Psi(x)$ is well defined and bounded. Hence it can be
extended to whole of $K_1$.

Next we prove that $\Psi$ is a $\rho$-morphism. For this, let $x,y\in E,a_i, b_j\in \mathcal A, g_i, h_j\in H_1,\; i=1,2,\ldots , m;  j=1,2,\dots,n; \, m, n\geq 1$. Then
\begin{align*}
&\left\langle \Psi(x)^*\Psi(y)\big(\sum_{j=1}^n \rho(b_j)Vh_j\big),\sum_{i=1}^m \rho(a_i)Vg_i\right\rangle \\
&=\left\langle \sum_{j=1}^n \Phi(yb_j)h_j,\sum_{i=1}^m\Phi(xa_i)g_i\right\rangle \\
                                                                    &=\sum _{j=1}^n\sum_{i=1}^m \langle  (\Phi(xa_i))^*\Phi(yb_j)h_j,g_i\rangle \\
                                                              &=\sum _{j=1}^n\sum_{i=1}^m \langle  \phi(\langle xa_i,yb_j\rangle )h_j,g_i\rangle \\
                                                        &=\sum _{j=1}^n\sum_{i=1}^m\langle V^*\rho(a_i)^* \rho(\langle x,y\rangle)\rho(b_j)Vh_j,g_i\rangle \\
                                  &=\left\langle  \rho(\langle x,y\rangle) \big(\sum_{j=1}^n\rho(b_j)Vh_j\big),\sum_{i=1}^m \rho(a_i)Vg_i\right\rangle.
\end{align*}
Thus $\Psi(x)^*\Psi(y)=\rho(\langle x,y\rangle)$ on the dense set
span($\rho(A)VH_1$) and hence they are equal on $K_1$.

Note that $K_2\subseteq H_2$. Let $W:=P_{K_2}$, the orthogonal projection onto $K_2$. Then $W^*:K_2\rightarrow H_2$ is the inclusion map. Hence $WW^*=I_{K_2}$.
That is $W$ is a co-isometry.

Finaly we give a representation for $\Phi$. For $x\in E$ and $h\in H_1$, we have
\begin{equation*}
W^*\Psi(x)Vh=\Psi(x)Vh=\Psi(x)(\rho(1)Vh)=\Phi(x)h.\qedhere
\end{equation*}
\end{proof}
\begin{defn}\label{minimalrepresentation}
 Let $\phi$ and $\Phi$ be as in Theorem \ref{stinespringhilbertmodules}. We say that a pair  of triples $\big((\rho, V, K_1),(\Psi, W, K_2)\big)$  is a \textbf{ Stinespring representation} of $(\phi, \Phi)$ if conditions (1)-(3)  of Theorem \ref{stinespringhilbertmodules} are satisfied. Such a representation  is said to be  \textbf{minimal}  if
\begin{itemize}
 \item [(a)] $K_1=[\rho(A)VH_1]$  and
\item [(b)] $K_2=[\Psi(E)VH_1]$.
\end{itemize}
\end{defn}
\begin{rmk}
Let $\phi$ and $\Phi$ be as in  Theorem
\ref{stinespringhilbertmodules}. The pair
$\big((\rho,V,K_1),(\Psi,W,K_2)\big)$ obtained  in the proof of
Theorem \ref{stinespringhilbertmodules},  is a minimal
representation for $(\phi,\Phi)$.
\end{rmk}
\begin{thm}\label{minimalequivalence}
Let $\phi$ and $\Phi$ be as in Theorem \ref{stinespringhilbertmodules}. Assume that $\big((\rho,V,K_1), (\Psi,W,K_2)\big)$ and $\big((\rho', V',K'_1), (\Psi',W',K'_2)\big)$ are minimal representations
for $(\phi,\Phi)$. Then there exists unitary operators $U_1:K_1\rightarrow K_1'$ and $U_2:K_2\rightarrow K_2'$ such that
\begin{enumerate}
 \item \label{unitary1} $U_1V=V',\; U_1\rho(a)=\rho'(a)U_1, \quad \text{for all}\; a\in \mathcal A$ and
\item \label{unitary2} $U_2W=W',\; U_2\Psi(x)=\Psi'(x)U_1,\quad \text{for all}\; x\in E$.
\end{enumerate}
That is, the following diagram commutes, for $a\in \mathcal A$ and $x\in E:$\\
$$\begin{diagram}
H_1 &\rTo^{V} &K_1      &\rTo^{\rho(a)}    &K_1       &\rTo^{\Psi(x)}   &K_2      &\lTo^{W}    &H_2\\
    &\rdTo_{V'}    &\dTo_{U_1} &                  &\dTo_{U_1}   &         &\dTo_{U_2} &\ldTo_{W'} \\
    &         &K'_1        &\rTo_{\rho'(a)}          &K'_1         &\rTo_{\Psi'(x)}         &K'_2
\end{diagram}$$
\end{thm}
\begin{proof}
Existence of the unitary map $U_1:K_1\rightarrow K_1'$ follows by
Theorem \cite[Theorem 4.2, page 46]{paulsencbmoa}. This can be
obtained as follows: First define $U_1:\text{span}(\rho(\mathcal
A)VH_1)\rightarrow \text{span}(\rho'(\mathcal A)V'H_1)$   by
\begin{equation*}
U_1 (\sum_{j=1}^n \rho(a_j)Vh_j)
:=\sum_{j=1}^n\rho'(a_j)V'h_j,\quad  a_j\in \mathcal A,\; h_j\in
H_1,j=1,\dots,n,\, n\geq 1,
\end{equation*}
which can be seen to be an onto isometry. Let us denote the
extension of $U_1$ to $K_1$ by $U_1$ itself. Then $U_1$ is unitary
and satisfies the conditions in (\ref{unitary1}).

Now define $U_2:\text{span}(\Psi(E)VH_1)\rightarrow
\text{span}(\Psi'(E)V'H_1)$ by
\begin{equation*}
U_2(\sum_{j=1}^n\Psi(x_j)Vh_j):=\sum_{j=1}^n\Psi'(x_j)V'h_j, \quad x_j\in E, h_j\in H_1, j=1,2,\dots, n,\, n\geq 1.
\end{equation*}
We claim that $U_2$ is well defined and can be extended to a unitary map. For this consider
\begin{align*}
 \left \|\sum_{j=1}^n\Psi'(x_j)V'h_j\right\|^2 &= \left\langle \sum_{j=1}^n\Psi'(x_j)V'h_j,\sum_{i=1}^n\Psi'(x_i)V'h_i\right\rangle\\
                                            &= \sum_{i,j=1}^n\langle \Psi'(x_j)V'h_j,\Psi'(x_i)V'h_i\rangle\\
                                            &= \sum_{i,j=1}^n\langle h_j,V'^*(\Psi'(x_j))^*\Psi'(x_i)V'h_i\rangle\\
                                            &= \sum_{i,j=1}^n\langle h_j,V'^*\rho'(\langle x_j,x_i\rangle)V'h_i\rangle\\
                                            &= \sum_{i,j=1}^n\langle h_j,V^*\rho(\langle x_j,x_i\rangle)Vh_i\rangle\\
                                            &= \sum_{i,j=1}^n\langle h_j,V^*(\Psi(x_j))^*\Psi(x_i)Vh_i\rangle\\
                                            &= \left\langle \sum_{j=1}^n\Psi(x_j)Vh_j,\sum_{i=1}^n\Psi(x_i)Vh_i\right\rangle\\
                                            &=\left\|\sum_{j=1}^n\Psi(x_j)Vh_j\right\|^2.
\end{align*}
This concludes that $U_2$ is well defined and an isometry. Hence it can be extended to whole of $K_2$, call the extension $U_2$ itself, and being onto it is a unitary.

Since $\big((\rho,V,K_1),(\Psi,W,K_2)\big)$ and $\big((\rho',V',K'_1),(\Psi',W',K'_2)\big)$ are representations for $(\phi,\Phi)$, it follows that
\begin{align*}
\Phi(x)=W^*\Psi(x)V&=W'^*\Psi'(x)V'=W'^*U_2\Psi(x)V\\
                                  &\Rightarrow (W^*-W'^*U_2)\Psi(x)V=0,
\end{align*}
equivalently, $(W^*-W'^*U_2)\Psi(x)Vh=0$ for all $h\in H_1$. Since $[\Psi(E)VH_1]=K_2$, it follows that $U_2W=W'$.

To prove the remaining  part of  (\ref{unitary2}), it is enough to
show  $U_2\Psi(x)=\Psi'(x)U_1$ on the dense set
span($\rho(\mathcal A)VH_1$). Let $a_j\in \mathcal A, h_j\in
H_1,j=1,2,\dots,n,\; n\geq 1$. Consider
\begin{align*}
 U_2\Psi(x)\big(\sum_{j=1}^n\rho(a_j)Vh_j\big)&= U_2\big(\sum_{j=1}^n\Psi(xa_j)V h_j\big)\quad  (\text{since}\; \Psi \; \text{is}\; \rho-\text{morphism})\\
                                              &= \sum_{j=1}^n\Psi'(xa_j)V' h_j\;  \\
                                              &= \Psi'(x)\big(\sum_{j=1}^n\rho'(a_j)V' h_j\big) \quad (\text{since}\; \Psi' \; \text{is}\; \rho'-\text{morphism})\\
                                              &= \Psi'(x)U_1\big(\sum_{j=1}^n\rho(a_j)V h_j\big). \qedhere
\end{align*}
\end{proof}
\begin{rmk}
Let $\big((\rho,V,K_1),(\Psi,W,K_2)\big)$ be a Stinespring
representation for $(\phi,\Phi)$. If $\phi$ is unital, then $V$ is
an isometry. If the representation is minimal, then $W$ is a
co-isometry by the proof of  Theorem
\ref{stinespringhilbertmodules} and (\ref{unitary2}) of Theorem
\ref{minimalequivalence}.
\end{rmk}

We give an example to illustrate our result.
\begin{eg}\label{eg1}
 Let $\mathcal A=\mathcal M_2(\mathbb C),\; H_1=\mathbb C^2, H_2=\mathbb C^8$ and $E=\mathcal A\oplus \mathcal A$. Define $\phi:\mathcal A\rightarrow \mathcal B(H_1)$ by

$ \phi \Big(\begin{pmatrix}
a_{11} & a_{12}  \\
a_{21} & a_{22}
\end{pmatrix}\Big)= \begin{pmatrix}
a_{11} &  \frac{a_{12}}{2}  \\
 \frac{a_{21}}{2} & a_{22}
\end{pmatrix}, \; \text{for all}\; a_{ij}\in \mathbb C, i,j=1,2.$

Let $D=\begin{pmatrix}
        1& \frac{1}{2}\\
\frac{1}{2}& 1
       \end{pmatrix}$. Then  $\phi(A)=D \circ A,\;\text{for all}\; A\in \mathcal A$, here $\circ$ denote the Schur product. As $D$ is positive, $\phi$
is a completely positive map (see \cite[Theorem 3.7, page 31]{paulsencbmoa} for details).
 Define $\Phi:E\rightarrow \mathcal B(H_1, H_2)$ by

$$ \Phi\big((a_{ij}\oplus (b_{ij})\big)=  \left(\begin{matrix}
                                \medskip
                                \frac{\sqrt{3}}{2}a_{11} & \frac{\sqrt{3}}{2}a_{12}\\
                                 \medskip
                                 \frac{\sqrt{3}}{2}a_{21} & \frac{\sqrt{3}}{2}a_{22} \\
                                  \medskip
                                 \frac{\sqrt{3}}{2}b_{11} & \frac{\sqrt{3}}{2}b_{12}\\
                                  \medskip
                                 \frac{\sqrt{3}}{2}b_{21} & \frac{\sqrt{3}}{2}b_{22} \\
                                  \medskip
                                  \frac{1}{2}a_{11} & -\frac{1}{2}a_{12}\\
                                   \medskip
                                 \frac{1}{2}a_{21} & -\frac{1}{2}a_{22} \\
                                   \medskip
                                  \frac{1}{2}b_{11} & -\frac{1}{2}b_{12}\\
                                  \medskip
                                 \frac{1}{2}b_{21} & -\frac{1}{2}b_{22}
                               \end{matrix}\right), \quad (a_{ij}),(b_{ij})\in \mathcal A, i,j=1,2.$$
It can be verified that $\Phi$ is a $\phi$-map.

Let $K_1=\mathbb C^4$ and $K_2= H_2$.
In this case $\rho:\mathcal A\rightarrow \mathcal B(K_1),\; V:H_1\mathcal \rightarrow K_1$ and $\Psi:E\rightarrow \mathcal B(K_1,K_2)$ are given by
\begin{align*}
 \rho((a_{ij}))&=\left(\begin{array}{cc}
                 (a_{ij})  & 0\\
                 0& (a_{ij})
                \end{array}\right), \quad (a_{ij})\in \mathcal A, \; i,j=1,2.\\
V&=\begin{pmatrix}
\medskip
  \frac{\sqrt{3}}{2}&0\\
\medskip
  0&\frac{\sqrt{3}}{2} \\
\frac{1}{2}& 0\\
\medskip
0& -\frac{1}{2}
\end{pmatrix} \\
\Psi \big((a_{ij})\oplus (b_{ij})\big)&=\begin{pmatrix}
                                        (a_{ij}) &0\\
                                        (b_{ij})&0\\
                                         0&(a_{ij})\\
                                         0&(b_{ij})
                                       \end{pmatrix},\quad (a_{ij}),(b_{ij})\in \mathcal A,\; i,j=1,2.
\end{align*}
It is easy to verify that $\Psi$ is a $\rho$-morphism and $\Phi((a_{ij})\oplus (b_{ij}))=W^*\Psi((a_{ij})\oplus (b_{ij}))V$, where $W=I_{H_2}$.
This example illustate Theorem (\ref{stinespringhilbertmodules}).

\end{eg}
\begin{rmk}
 Note that in Example \ref{eg1}, there does not exists an $x_0\in E$ with the property that $\Phi(x_0)\Phi(x_0)^*=I_{H_2}$, which is an assumption in  Theorem \ref{asdithm}.
\end{rmk}

 \bibliographystyle{amsplain}
 \bibliography{ramesh2}

\providecommand{\bysame}{\leavevmode\hbox to3em{\hrulefill}\thinspace}
\providecommand{\MR}{\relax\ifhmode\unskip\space\fi MR }
\providecommand{\MRhref}[2]{%
  \href{http://www.ams.org/mathscinet-getitem?mr=#1}{#2}
}
\providecommand{\href}[2]{#2}
\begin{thebibliography}{1}

\bibitem{asadi}
Mohammad~B. Asadi, \emph{Stinespring's theorem for Hilbert $C^*$-modules}, J.
  Operator Theory \textbf{62} (2008), no.~2, 235--238.

\bibitem{ringkad1}
Richard~V. Kadison and John~R. Ringrose, \emph{Fundamentals of the theory of
  operator algebras. {V}ol. {I}}, Graduate Studies in Mathematics, vol.~15,
  American Mathematical Society, Providence, RI, 1997, Elementary theory,
  Reprint of the 1983 original. \MR{MR1468229 (98f:46001a)}

\bibitem{lance}
E.~C. Lance, \emph{Hilbert {$C\sp *$}-modules}, London Mathematical Society
  Lecture Note Series, vol. 210, Cambridge University Press, Cambridge, 1995, A
  toolkit for operator algebraists. \MR{MR1325694 (96k:46100)}

\bibitem{paulsencbmoa}
Vern Paulsen, \emph{Completely bounded maps and operator algebras}, Cambridge
  Studies in Advanced Mathematics, vol.~78, Cambridge University Press,
  Cambridge, 2002. \MR{MR1976867 (2004c:46118)}

\bibitem{wfstinespring}
W.~Forrest Stinespring, \emph{Positive functions on {$C\sp *$}-algebras}, Proc.
  Amer. Math. Soc. \textbf{6} (1955), 211--216. \MR{MR0069403 (16,1033b)}

\end{thebibliography}
\end{document}